\documentclass[12pt,a4paper]{amsart}
\usepackage{amsfonts}
\usepackage{mathrsfs}
\usepackage{amsthm,amsxtra}
\usepackage{amssymb}
\usepackage{amsmath}
\usepackage{amscd}
\usepackage{hyperref}
\hypersetup{
colorlinks=true,
linkcolor=blue,
citecolor=gray,
}
\usepackage[latin2]{inputenc}
\usepackage{t1enc}
\usepackage[mathscr]{eucal}
\usepackage{indentfirst}
\usepackage{graphicx}
\graphicspath{ {images/} }
\usepackage{subfig}
\usepackage{enumerate}
\usepackage{graphics}
\usepackage{pict2e}
\usepackage{epic}
\numberwithin{equation}{section}
\usepackage[margin=2.9cm]{geometry}
\usepackage{epstopdf} 
\usepackage{longtable}
\usepackage{enumerate}
\usepackage[inline]{enumitem}
\usepackage{tikz-cd}
\usepackage{tikz}
\usepackage{color,soul}

\usepackage{todonotes}

\DeclareMathOperator{\Mod}{Mod}

\DeclareMathOperator{\Homeo}{Homeo}

\newcommand{\SSS}{\mathbb{S}}

\theoremstyle{plain}
\newtheorem{theorem}{Theorem}
\newtheorem{proposition}[theorem]{Proposition}
\newtheorem{lemma}[theorem]{Lemma}

\theoremstyle{definition}

\newtheorem{conjecture}[theorem]{Conjecture}

\theoremstyle{remark}

\newtheorem*{remark*}{Remark}

\begin{document}

\title[A note on the Burnside problem for $\Homeo(M)$]{A note on the Burnside problem for homeomorphism groups of manifolds}
	
\author{Donggyun Seo}
\address{Institute of Mathematical Sciences \\ Chungnam National University \\ Daejeon, South Korea}
\email{seodonggyun@cnu.ac.kr \\ seodonggyun7@gmail.com}
\urladdr{https://www.seodonggyun.com}

\begin{abstract}
    This note studies the Burnside problem for homeomorphism groups of compact connected manifolds.
    For surfaces, we prove that the identity component of the homeomorphism group is torsion-free precisely when the surface is not the sphere, torus, projective plane, or Klein bottle.
    An extension argument based on the Tits alternative for mapping class groups then implies that every finitely generated periodic subgroup of the full homeomorphism group is finite for all surfaces outside this exceptional list, recovering and extending a theorem of Guelman and Liousse to non-orientable surfaces.
    For the circle, we prove that every finitely generated periodic subgroup of its homeomorphism group is finite and cyclic.
    We close with remarks on manifolds with boundary and open questions on the Burnside problem for hyperbolic three-manifolds and doubled handlebodies.
\end{abstract}

\maketitle

\section*{Burnside problem for surface homeomorphisms}

Let $\Sigma$ be a real, compact, connected surface (orientable or non-orientable, possibly with boundary).
The main object of study in this note is the group $\Homeo(\Sigma) = \Homeo(\Sigma, \partial\Sigma)$ of homeomorphisms of $\Sigma$ fixing the boundary pointwise.
If $\Sigma$ is orientable, let $\Homeo^+(\Sigma)$ denote the subgroup of orientation-preserving elements; if $\Sigma$ is non-orientable, set $\Homeo^+(\Sigma) = \Homeo(\Sigma)$.
Let $\Homeo_0(\Sigma)$ denote the identity component of $\Homeo^+(\Sigma)$.
We write $\mathbb{S}^2$, $T^2$, $\mathbb{R}P^2$, and $K$ for the sphere, torus, real projective plane, and Klein bottle, respectively.

\begin{proposition} \label{prop:torsion_free}
For a compact connected surface $\Sigma$, the following are equivalent:
\begin{enumerate}
    \item \label{enum:1} $\Sigma \notin \{ \mathbb{S}^2, T^2, \mathbb{R}P^2, K \}$,
    \item \label{enum:2} $\Homeo_0(\Sigma)$ is torsion-free, and
    \item \label{enum:3} $\Homeo_0(\Sigma)$ is contractible.
\end{enumerate}
\end{proposition}

The equivalence of conditions \eqref{enum:1} and \eqref{enum:3} was established by Kneser \cite{MR1544816}, Hamstrom \cite{MR170334, MR184234, MR202140}, and Luke--Mason \cite{MR301693}.
Although the equivalence of \eqref{enum:1} and \eqref{enum:2} is likely well known, the author could not discover a proof in the literature.
This gap motivated the present note.
The implication \eqref{enum:2}$\Rightarrow$\eqref{enum:1} can be derived by exhibiting a nontrivial torsion element in $\Homeo_0(\Sigma)$ for each $\Sigma \in \{\mathbb{S}^2, \mathbb{R}P^2, T^2, K\}$.

\begin{enumerate}
    \item \textit{Sphere.} Identify $\mathbb{S}^2 = \{v \in \mathbb{R}^3 \mid \|v\| = 1\}$.
    The continuous, injective homomorphism $\mathrm{SO}(3) \to \Homeo_0(\mathbb{S}^2)$ given
    by matrix multiplication is well defined since $\mathrm{SO}(3)$ is connected. For
    $\theta = q\pi$ with $q \in \mathbb{Q} \setminus 2\mathbb{Z}$, the rotation
    \[
        v \mapsto \begin{pmatrix} 1 & 0 & 0 \\ 0 & \cos\theta & -\sin\theta \\
        0 & \sin\theta & \cos\theta \end{pmatrix} v
    \]
    is a nontrivial element of finite order in $\Homeo_0(\mathbb{S}^2)$; writing $q = p/n$
    in lowest terms, its order is $2n$ if $p$ is odd and $n$ if $p$ is even.

    \item \textit{Real projective plane.} Identify $\mathbb{R}P^2 = \mathbb{S}^2/(v \sim -v)$.
    Since every $A \in \mathrm{SO}(3)$ commutes with the antipodal map, it descends to a
    homeomorphism $\bar{A}([v]) = [Av]$ of $\mathbb{R}P^2$. This defines a continuous,
    injective homomorphism $\mathrm{SO}(3) \to \Homeo_0(\mathbb{R}P^2)$, where injectivity
    follows from $-I \notin \mathrm{SO}(3)$. The same rotation as in the spherical case,
    \[
        \bar{A}_\theta([v]) = \left[\begin{pmatrix} 1 & 0 & 0 \\ 0 & \cos\theta & -\sin\theta \\
        0 & \sin\theta & \cos\theta \end{pmatrix} v\right]
    \]
    with $\theta = q\pi$ and $q \in \mathbb{Q} \setminus 2\mathbb{Z}$, therefore yields a
    nontrivial torsion element of $\Homeo_0(\mathbb{R}P^2)$.

    \item \textit{Torus.} Identify $T^2 = \mathbb{R}^2/\mathbb{Z}^2$. Since $\mathbb{R}^2$
    is connected, the translation action $v \cdot [x] = [x+v]$ maps into $\Homeo_0(T^2)$.
    For any $v = (p/q,\, r/s) \in \mathbb{Q}^2 \setminus \mathbb{Z}^2$, the translation
    $\tau_v \colon [x] \mapsto [x+v]$ is a nontrivial homeomorphism of finite order
    $\mathrm{lcm}(q,s)$, hence a nontrivial torsion element of $\Homeo_0(T^2)$.

    \item \textit{Klein bottle.} Identify $K = \mathbb{R}^2/\Gamma$, where $\Gamma$ is
    generated by
    \[
        \gamma_1 \colon (x,y) \mapsto (x+1,\, y), \qquad
        \gamma_2 \colon (x,y) \mapsto (-x,\, y+1).
    \]
    The map $\varphi \colon K \to K$ induced by $(x,y) \mapsto (x,\, y + \frac{1}{2})$
    is well-defined, satisfies $\varphi^2 = id$, and is isotopic to the identity
    via the homotopy $\varphi_t \colon (x,y) \mapsto (x,\, y + \tfrac{t}{2})$, $t \in [0,1]$.
    Hence $\varphi$ is a nontrivial torsion element of $\Homeo_0(K)$.
\end{enumerate}
In each case $\Homeo_0(\Sigma)$ contains a nontrivial torsion element.

To prove the implication \eqref{enum:1}$\Rightarrow$\eqref{enum:2} in Proposition \ref{prop:torsion_free}, we recall several basic properties of surface braid groups and mapping class groups.
Most of the material between Proposition \ref{prop:torsion_free} and its proof is adapted from the survey paper of Guaschi--Juan-Pineda \cite{MR3382024}.

Let $D_n(\Sigma)$ denote the space of unordered $n$-tuples of distinct points in $\Sigma$.
More precisely, $D_n(\Sigma)$ is the quotient of
\[
\tilde{D}_n(\Sigma) := \{ (p_1, \dots, p_n) \in \Sigma^n \mid p_i \neq p_j \text{ for } i \neq j \}
\]
under the map $(p_1, \dots, p_n) \mapsto \{p_1, \dots, p_n\}$.
We equip $D_n(\Sigma)$ with the quotient topology induced from $\tilde{D}_n(\Sigma)$.

Fix a positive integer $n$ and a finite subset $P \subset \Sigma$ with $|P| = n$.
The \emph{$n$-strand braid group} of $\Sigma$ based at $P$, denoted $B_n(\Sigma)$ or $B_n(\Sigma, P)$, is defined as the fundamental group $\pi_1(D_n(\Sigma), P)$.
It is known from Fadell--Neuwirth \cite{MR141126} and van Buskirk \cite{MR189013} that $B_n(\Sigma)$ is torsion-free if and only if $\Sigma \notin \{\mathbb{S}^2, \mathbb{R}P^2\}$.

The natural diagonal action of $\Homeo^+(\Sigma)$ on $D_n(\Sigma)$ induces a map
\[
\Psi: \Homeo^+(\Sigma) \to D_n(\Sigma), \quad f \mapsto f(P),
\]
which is a fiber bundle with fiber $\Homeo^+(\Sigma, P)$, the subgroup fixing $P$ setwise.
This gives rise to the following long exact sequence:
\begin{center}
\begin{tikzcd}
& \ar[d, phantom, ""{coordinate, name = PH1}] & \cdots \ar[r] & \pi_2(D_n(\Sigma), P) \ar[dll, to path = {-- ([xshift=2ex]\tikztostart.east) |- (PH1) [near end]\tikztonodes -| ([xshift=-2ex]\tikztotarget.west) -- (\tikztotarget)}] \\
 & \pi_1(\Homeo^+(\Sigma, P), id) \ar[r] \ar[d, phantom, ""{coordinate, name = PH}] & \pi_1(\Homeo^+(\Sigma), id) \ar[r] & \pi_1(D_n(\Sigma), P) \ar[dll, to path = {-- ([xshift=2ex]\tikztostart.east) |- (PH) [near end]\tikztonodes -| ([xshift=-2ex]\tikztotarget.west) -- (\tikztotarget)}] \\
& \pi_0(\Homeo^+(\Sigma, P), id) \ar[r] & \pi_0(\Homeo^+(\Sigma), id) \ar[r] & 1
\end{tikzcd}
\end{center}

On the other hand, Hamstrom \cite{MR202140} proved that $\pi_1(\Homeo^+(\Sigma), id)$ is trivial whenever $\Sigma \notin \{\mathbb{S}^2, T^2, \mathbb{R}P^2, K\}$.
This yields the short exact sequence
\[
1 \to B_n(\Sigma, P) \to \Mod(\Sigma, P) \to \Mod(\Sigma) \to 1.
\]
For $n = 1$, this coincides with the classical Birman exact sequence.

We summarize these facts in the following lemma.

\begin{lemma} \label{lem:tech}
For all $n \geq 1$, the following hold:
\begin{enumerate}
\item \label{num:tech1} $B_n(\Sigma)$ is torsion-free if and only if $\Sigma \notin \{ \mathbb{S}^2, \mathbb{R}P^2 \}$.
\item \label{num:tech2} If $\Sigma \notin \{\mathbb{S}^2, T^2, \mathbb{R}P^2, K\}$, then the following exact sequence holds:
\[
1 \to B_n(\Sigma, P) \to \Mod(\Sigma, P) \to \Mod(\Sigma) \to 1.
\]
\end{enumerate}
\end{lemma}

\begin{proof}[Proof of Proposition \ref{prop:torsion_free}]
Suppose there exists a non-identity torsion element $f \in \Homeo_0(\Sigma)$.
Choose a point $x \in \Sigma$ in the support of $f$, and let $P$ be the $\langle f \rangle$-orbit of $x$.
Then $P$ contains finitely many points, with at least two elements.
Since $f$ cyclically permutes the points of $P$, it represents a nontrivial element of $\Mod(\Sigma, P)$.

However, $f$ represents the identity element in $\Mod(\Sigma)$.
By Lemma \ref{lem:tech}\eqref{num:tech2}, $f$ therefore corresponds to a finite-order element in $B_n(\Sigma, P)$.
This contradicts Lemma \ref{lem:tech}\eqref{num:tech1}, which asserts that $B_n(\Sigma, P)$ is torsion-free.
Hence, no such $f$ exists, and $\Homeo_0(\Sigma)$ is torsion-free.
\end{proof}

An element of a group is called \emph{periodic} if it has finite order, and a group is called \emph{periodic} if all of its elements are periodic.
Burnside asked whether every finitely generated periodic group is finite.
Many counterexamples have since been constructed by Golod--Shafarevich \cite{MR161852}, Adian--Novikov \cite{MR240178, MR240179, MR240180}, and others.
Nevertheless, one can often reduce Burnside-type questions to subgroups of specific groups.
For instance, Hurtado--Kocsard--Rodriguez-Hertz \cite{MR4173155} proved that every finitely generated periodic subgroup of smooth area-preserving homeomorphisms of $\SSS^2$ is finite, and Guelman--Liousse \cite{MR3158049} showed Burnside property holds for measure-preserving homeomorphisms on $T^2$.

By Tits alternative for mapping class groups of orientable surfaces \cite{MR800253}, every periodic subgroup is virtually solvable.
In the non-orientable case (excluding $\mathbb{R}P^2$ and $K$), the orientable double cover induces an embedding of the mapping class group into that of an orientable surface \cite{MR300288, MR2675928, MR4729813}, implying that Tits alternative still holds.
Since the mapping class groups of $\mathbb{R}P^2$ and $K$ are finite, it follows that all mapping class groups of surfaces satisfy Tits alternative.

\begin{lemma} \label{lem:tits}
Let $G$ be a group satisfying the Tits alternative, i.e., every finitely generated subgroup of $G$ is either virtually solvable or contains a free subgroup of rank $2$. 
Then every finitely generated periodic subgroup of $G$ is finite.
\end{lemma}

\begin{proof}
Let $K \leq G$ be a finitely generated periodic subgroup.
By the Tits alternative, $K$ is either virtually solvable or contains a free subgroup $F \cong F_2$.
The latter is impossible, since $F_2$ contains elements of infinite order, contradicting $K$ being periodic.
Hence $K$ is virtually solvable, so there exists a solvable subgroup $S \leq K$ of finite index.

It remains to show that $S$ is finite.
Since $[K:S] < \infty$ and $K$ is finitely generated, the Schreier transversal lemma implies that $S$ is finitely generated.
As a subgroup of $K$, the group $S$ is also periodic.
We show by induction on the derived length $d$ of $S$ that every finitely generated periodic solvable group is finite.

If $d = 1$, then $S$ is abelian, and by the classification of finitely generated abelian groups, $S \cong \mathbb{Z}^r \oplus T$ for some $r \geq 0$ and finite group $T$.
Since $S$ is periodic, we have $r = 0$, so $S = T$ is finite.

If $d \geq 2$, the abelianization $S/S'$ is a finitely generated periodic abelian group, hence finite by the case $d=1$.
Since $[S:S'] < \infty$ and $S$ is finitely generated, the Schreier transversal lemma implies that $S'$ is finitely generated.
Since $S'$ is a subgroup of $S$, it is periodic, and it is solvable of derived length $d-1$.
By the induction hypothesis, $S'$ is finite, and therefore $|S| = [S:S'] \cdot |S'| < \infty$.

Hence $S$ is finite, and consequently $|K| = [K:S] \cdot |S| < \infty$.
\end{proof}

Consequently, every finitely generated periodic subgroup of $\mathrm{Mod}(\Sigma)$ is finite.
The same conclusion holds for the extended mapping class group $\mathrm{Mod}^{\pm}(\Sigma)$.
Indeed, let $H \leq \mathrm{Mod}^{\pm}(\Sigma)$ be a finitely generated periodic subgroup.
Since $\mathrm{Mod}(\Sigma)$ has index $2$ in $\mathrm{Mod}^{\pm}(\Sigma)$, the intersection $H_0 := H \cap \mathrm{Mod}(\Sigma)$ has index at most $2$ in $H$, hence is finite index.
Since $H$ is finitely generated and $[H : H_0] < \infty$, the Schreier transversal lemma implies that $H_0$ is 
finitely generated.
As a subgroup of $H$, the group $H_0$ is also periodic.
Since $H_0 \leq \mathrm{Mod}(\Sigma)$, Lemma \ref{lem:tits} implies that $H_0$ is finite. 
Therefore
\[
|H| = [H : H_0] \cdot |H_0| < \infty,
\]
so $H$ is finite.

\begin{lemma} \label{lem:group}
Let \[1 \to A \rightarrow B \xrightarrow{\pi} C \to 1\] be a short exact sequence of groups.
If $A$ and $C$ both satisfy the Burnside property (that is, every finitely generated periodic subgroup is finite), then so does $B$.
\end{lemma}

\begin{proof}
Let $H \leq B$ be a finitely generated periodic subgroup.
Since $\pi(H)$ is a finitely generated periodic subgroup of $C$ and $C$ satisfies the Burnside property, $\pi(H)$ is finite.
Set $K := H \cap \ker \pi = H \cap A$, so that $K$ is a normal subgroup of $H$ with $H/K \cong \pi(H)$.
In particular, $[H : K] = |\pi(H)| < \infty$.

Since $K$ is a subgroup of finite index in the finitely generated group $H$, it is itself finitely generated by the Schreier transversal lemma.
Moreover, every element of $K \leq \ker\pi \cong A$ has finite order since $H$ is periodic.
Thus $K$ is a finitely generated periodic subgroup of $A$, and since $A$ satisfies the Burnside property, $K$ is finite.

Since $H$ is an extension of the finite group $\pi(H) \cong H/K$ by the finite group $K$, we conclude that $H$ is finite.
\end{proof}

As a corollary of Lemma~\ref{lem:group}, the Burnside property is preserved under group extensions with torsion-free kernel.
Since $\Homeo_0(\Sigma)$ is torsion-free for $\Sigma \notin \{\mathbb{S}^2, T^2, \mathbb{R}P^2, K\}$ and $\Mod^\pm(\Sigma)$ satisfies the Burnside property, the short exact sequence
\[
    1 \to \Homeo_0(\Sigma) \to \Homeo(\Sigma) \to \Mod^\pm(\Sigma) \to 1
\]
immediately implies that $\Homeo(\Sigma)$ satisfies the Burnside property as well, yielding the following theorem.

\begin{theorem} \label{thm:burnside}
If $\Sigma \notin \{ \mathbb{S}^2, T^2, \mathbb{R}P^2, K \}$, then every finitely generated periodic subgroup of $\Homeo(\Sigma)$ is finite.
\end{theorem}

Note that Guelman--Liousse \cite{MR3712338} established the Burnside property for groups of homeomorphisms of compact orientable surfaces. More precisely, they proved that every finitely generated periodic subgroup of the homeomorphism group of an orientable compact surface, except for the sphere and the torus, is finite. Their approach relies on a combination of dynamical and geometric arguments, including the analysis of group actions with finite orbits and rigidity phenomena for periodic homeomorphisms. Our result recovers this conclusion through a different argument, based on the torsion-freeness of the identity component and a general extension principle.

The four exceptional surfaces (the sphere, the torus, the projective plane, and the Klein bottle) lie outside the scope of Theorem \ref{thm:burnside}, and the Burnside problem for their full homeomorphism groups remains incompletely resolved.
For the sphere, Hurtado--Kocsard--Rodr\'iguez-Hertz \cite{MR4173155} proved the Burnside property for the group of area-preserving analytic diffeomorphisms of $\mathbb{S}^2$, but the purely topological case of $\Homeo(\mathbb{S}^2)$ is open.
Partial progress has been made: Conejeros \cite{MR3892240} proved that every finitely generated group of homeomorphisms of $\mathbb{S}^2$ whose elements all have finite order that is a power of $2$, with a uniform bound on the orders, is finite, and obtained a parallel result for area-preserving homeomorphisms.
For the torus, Guelman--Liousse \cite{MR3158049} proved that every finitely generated periodic group of homeomorphisms of $T^2$ preserving a probability measure is finite, but the measure-free case is open.
For the real projective plane and the Klein bottle, no analogue of these results appears in the literature, and the Burnside problem for $\Homeo(\mathbb{R}P^2)$ and $\Homeo(K)$ appears to be entirely open.
The difficulty in all four cases is precisely that $\Homeo_0(\Sigma)$ is not torsion-free, so the extension argument of Lemma \ref{lem:group} does not apply, and new dynamical or algebraic methods seem to be required.

\begin{conjecture}
    For each $\Sigma \in \{\mathbb{S}^2, T^2, \mathbb{R}P^2, K\}$, every finitely generated periodic subgroup of $\Homeo(\Sigma)$ is finite.
\end{conjecture}

\section*{Burnside problem for circle homeomorphisms}

Let us focus on the circle $\mathbb{S}^1$ and the compact interval $I = [0,1]$.
The group $\Homeo(I) = \Homeo^+(I)$ is torsion-free (see \cite{MR3560661}), so every finitely generated periodic subgroup is trivial.
The group $\Homeo^+(\mathbb{S}^1)$, by contrast, contains torsion: any rotation of finite order provides an example.
Nevertheless, finitely generated periodic subgroups remain highly constrained.

\begin{theorem} \label{thm:circle}
Every finitely generated periodic subgroup of $\Homeo(\mathbb{S}^1)$ is finite and cyclic.
\end{theorem}

\begin{proof}
By Lemma \ref{lem:group}, we need need only to check that every finitely generated periodic subgroup of $\Homeo^+(\mathbb{S}^1)$ is finite.
Let $G = \langle f_1, \ldots, f_k \rangle \leq
\Homeo^+(\mathbb{S}^1)$ be finitely generated and periodic, with $f_i^{n_i} = id$ for each $i$.

We first show that every non-identity element of $G$ has no fixed point.
Let $g \in G$ with $g \neq id$. Since $G$ is periodic, $g^m = id$ for some $m \geq 1$.
Suppose for contradiction that $g(x_0) = x_0$ for some $x_0 \in \mathbb{S}^1$.
The restriction of $g$ to the open arc $\mathbb{S}^1 \setminus \{x_0\} \cong \mathbb{R}$ is then an orientation-preserving homeomorphism of $\mathbb{R}$.
Such a homeomorphism is either the identity or fixed-point-free; in the latter case it satisfies $h(x) > x$ for all $x \in \mathbb{R}$ or $h(x) < x$ for all $x \in \mathbb{R}$ (since $h - id$ is continuous and nowhere zero, hence of constant sign).
In either sub-case, every orbit under $h$ is strictly monotone, so no orbit is finite, contradicting $g^m = id$.
Hence $g$ must restrict to the identity on $\mathbb{S}^1 \setminus \{x_0\}$, giving $g = id$, a contradiction.
Therefore $g$ has no fixed point on $\mathbb{S}^1$.

Since every non-identity element of $G$ has no fixed point, $G$ acts
freely on $\mathbb{S}^1$. By \cite[Theorem 2.2.32]{MR2809110}, $G$ is isomorphic to a subgroup of $\mathrm{SO}(2)$.
Because $\mathrm{SO}(2)$ is abelian and $G$ is periodic, $G$ is a finite abelian group.

It remains to show $G$ is cyclic. By \cite[Lemma~3.1]{MR3813208}, every finite subgroup of $\Homeo^+(\mathbb{S}^1)$ is cyclic.
The proof is as follows. Since every non-identity element of the finite group $G$ has no fixed point, $G$ acts freely on $\mathbb{S}^1$.
Fix any $x_0 \in \mathbb{S}^1$; since the stabilizer of $x_0$ in $G$ is trivial, the orbit $\mathcal{O} = G \cdot x_0$ has exactly $|G|$ points.
Label them $x_0, x_1, \ldots, x_{|G|-1}$ in their cyclic order inherited from $\mathbb{S}^1$.
Every element of $G$ sends $\mathcal{O}$ to itself by an orientation-preserving bijection, hence acts on $\mathcal{O}$ as a cyclic permutation $x_i \mapsto x_{i+\varphi(g) \pmod{|G|}}$ for a unique $\varphi(g) \in \mathbb{Z}/|G|\mathbb{Z}$.
The map $\varphi \colon G \to \mathbb{Z}/|G|\mathbb{Z}$ is a group homomorphism.
It is injective: if $\varphi(g) = 0$ then $g$ fixes every point of $\mathcal{O}$, in particular $g(x_0) = x_0$, so $g = id$ by the free action.
Since $|G| = |\mathbb{Z}/|G|\mathbb{Z}|$, $\varphi$ is an isomorphism, giving $G \cong \mathbb{Z}/|G|\mathbb{Z}$.
\end{proof}

\section*{Remarks and questions}

By a theorem of Newman \cite{N31} (see also Smith \cite{MR4128}, Dress \cite{MR238353}, and Bredon \cite{MR413144}), any periodic homeomorphism of a (finite-dimensional) compact connected manifold that fixes an open set pointwise must be the identity.
Note $\Homeo(M)$ is the group of \emph{pointwise boundary-fixing} homeomorphisms of $M$.
By Newman's theorem, we have the following.

\begin{proposition}
    If $M$ is a finite-dimensional compact connected manifold with boundary, then $\Homeo(M)$ is torsion-free.
\end{proposition}

\begin{proof}
    Let $f \in \Homeo(M)$ be a periodic homeomorphism.
    Define a collar $N = \partial M \times [0,1]$, and let $\partial_0 N = \partial M \times \{0\}$.
    Let us glue $M$ and $N$ along $\partial M$ and $\partial_0N$:
    \[
        M' = M \cup_{\partial M \,=\, \partial_0 N} N,
    \]
    which is homeomorphic to $M$.
    Extend $f$ to $f_{M'}$ by setting
    \[
        f_{M'} = \begin{cases} f & \text{on } M, \\ id_N & \text{on } N. \end{cases}
    \]
    Then $f_{M'}$ is periodic and fixes the open set $\partial M \times (0, 1) \subset N$ pointwise.
    By Newman's theorem, $f_{M'} = id_{M'}$, and hence $f = f_{M'}|_M = id_M$.
\end{proof}

So the Burnside problem for $\Homeo(M)$ is interesting when $M$ is closed (i.e. without boundary).
The most general conjecture along the way is the following.

\begin{conjecture}
    The homeomorphism group of a finite-dimensional closed connected manifold satisfies the Burnside property, i.e., every finitely generated periodic subgroup is finite.
\end{conjecture}

Before closing this note, we raise two questions.
The first concerns the Burnside problem for closed hyperbolic $3$-manifolds.
If $M$ is a closed hyperbolic $3$-manifold, then Mostow rigidity \cite{MR236383} implies that $\operatorname{Mod}(M)$ is finite.
By Lemma \ref{lem:group}, it therefore suffices to ask whether $\operatorname{Homeo}_0(M)$ satisfies the Burnside property, as a positive answer would imply that the full group $\operatorname{Homeo}(M)$ does as well.

The second concerns doubled handlebodies.
A doubled handlebody is the closed $3$-manifold obtained by doubling a handlebody along its boundary.
It is well known, by work of Whitehead \cite{MR1575455}, Laudenbach \cite{MR356056}, and Brendle--Broaddus--Putman \cite{MR4557874}, that its mapping class group is virtually isomorphic to $\operatorname{Out}(F_n)$; consequently, the Tits alternative holds for this group by a theorem of Bestvina--Feighn--Handel \cite{MR1765705, MR2150382}.
This motivates the following question: does the identity component $\operatorname{Homeo}_0(M)$ of the homeomorphism group of a doubled handlebody satisfy the Burnside property?

\section*{Acknowledgments}

The author is grateful to Sang-hyun Kim for reading an early draft of this note
and for helpful discussions. The author also thanks Nancy Guelman for reading
a preliminary version and for her valuable comments and suggestions.

\bibliographystyle{amsalpha}
\bibliography{references}

\end{document}